\documentclass[12pt]{amsart}

\usepackage{latexsym,amsfonts,amsmath,epsfig,tabularx,amsthm,dsfont,mathrsfs}

\usepackage[usenames,dvipsnames,svgnames,table]{xcolor}
\usepackage{graphicx}
\usepackage{pst-math,pst-xkey}
\usepackage{url}
\usepackage{hyperref}

\newtheorem{thm}{Theorem} 

\newtheorem{lem}[thm]{Lemma}

\newcommand\R{{\mathbb{R}}}
\newcommand\N{\mathbb{N}}

\newcommand{\cP}{\mathcal{P}}

\newcommand{\bD}{\mathbb{D}}

\newcommand{\D}{D_{\cP}}
\newcommand{\I}{[0,1)}
\newcommand{\norm}{\left\|\D|L_{2}\right\|}

\newcommand{\ud}{\, \mathrm{d}}

\title{An improved lower bound for the $L_{2}$-discrepancy}

\date{}

\author{Aicke Hinrichs} 
\address{Institute of Analysis, University Linz}
\email{aicke.hinrichs@jku.at}

\author{Gerhard Larcher}
\address{ Institute of Financial Mathematics and Applied Number Theory, University Linz}
\email{gerhard.larcher@jku.at}

\thanks{The second author is partially supported by the Austrian Science Fund (FWF), Project F5507-N26, which is a part of the Special Research Program ``Quasi-Monte Carlo Methods: Theory and Applications''.}

\begin{document}

\begin{abstract}
We give an improved lower bound for the $L_2$-discrepancy of finite point sets in the unit square.
\end{abstract}

\maketitle


\section{Introduction and main result}

Let $\cP=\{p_{1},p_{2},\ldots,p_{N}\}$ be a finite point set in the $d$-dimensional unit cube $\I^{d}$. For a rectangular box $B_{x}=[0,x_{1})\times\cdots\times[0,x_{d})$ in $\I^{d}$ let the discrepancy function be defined as
\begin{equation*}
 \D(x) := \sum_{i=1}^{N}\chi_{B_x}(p_i)-N\cdot x_1\cdots x_{d},
\end{equation*}
where $\chi_{B_x}$ denotes the characteristic function of the set $B_{x}$ and $x:=(x_{1},\ldots,x_{d})$. So the discrepancy function measures the deviation of the number of points of $\cP$ in the box $B_{x}$ from the fair number of points $N\cdot x_{1}\cdots x_{d}=:N|B_{x}|$.

The $L_{2}$-discrepancy of $\cP$ is the $L_2$-norm of $\D(x)$, given by
\begin{equation*}
 \norm=\left(  \int_{\I^{d}} \D^{2}(x)\ud x \right)^{\frac{1}{2}}.
\end{equation*}
By a celebrated result of K.F. Roth \cite{R} (see also \cite{KN}) it is known, that for every dimension $d$ there are constants $0<c_d<c_{d}^{\prime}$ such that
\begin{itemize}
 \item[(i)] for all $N$-point sets $\cP$ in $\I^{d}$ we have 
\begin{equation*}
 \norm\geq c_{d}\cdot\left(\log N\right)^{\frac{d-1}{2}}
\end{equation*}
\end{itemize}
and
\begin{itemize}
 \item[(ii)] for any $N\ge 2$, there exists an  $N$-point set $\cP$ in $\I^{d}$ such that 
 \begin{equation*}
  \norm\leq c_{d}^{\prime}\cdot\left(\log N\right)^{\frac{d-1}{2}}.
 \end{equation*}
\end{itemize}
Hence, it makes sense to ask for the positive real $\bar{c}_{d}$ which is defined by
\begin{equation*}
 \bar{c}_{d}:=\inf_{N\ge 2} \inf_{\# \cP = N} \norm\cdot\left(\log N\right)^{-\frac{d-1}{2}},
\end{equation*}
where the second  infimum is taken over all $N$-point sets $\cP$ in $\I^{d}$. 

In this work we are mainly interested in dimension $2$, i.e., in $\bar{c}_2$. 
To discuss known results, let us also introduce the constants
\begin{equation*}
 \bar{a}_{d}:=\liminf_{N\to\infty} \inf_{\# \cP = N} \norm\cdot\left(\log N\right)^{-\frac{d-1}{2}}
\end{equation*}
and
\begin{equation*}
 \bar{b}_{d}:=\limsup_{N\to\infty} \inf_{\# \cP = N} \norm\cdot\left(\log N\right)^{-\frac{d-1}{2}},
\end{equation*}
which are the interesting constants if we want to study the asymptotic behavior of the $L_2$-discrepancy for $N\to \infty$.
Obviously, 
\begin{equation*}
 \bar{c}_{d} \le \bar{a}_{d} \le \bar{b}_{d}.
\end{equation*}

The best constructions known so far  yield the upper bound
\begin{equation*}
 \bar{a}_{2}  \leq 0.17907\ldots
\end{equation*}
This is from a recent construction in \cite{FPPS} using generalized scrambled Hammersley point sets.
There is numerical evidence obtained in \cite{BTR} that the Fibonacci lattices are slightly better and give 
\begin{equation*}
 \bar{a}_{2}  \leq 0.176006\ldots
\end{equation*}
But until now this is not proved.

The best lower bounds known so far are
\begin{equation}\label{eq:corest1}
  0.038925\ldots \le \bar{b}_{2} 
\end{equation}
and 
\begin{equation}\label{eq:corest2}
 0.03276\ldots\leq\bar{c}_2.
\end{equation}

These bounds are shown in \cite{HM}.
Actually, there the first lower bound was even claimed to hold for $\bar{c}_2$. 
However, there is a small inaccuracy in the proof of the lower bound in \cite{HM}: At the end of Section~3, when the optimal value for the parameter $\gamma$ is determined, in the paper the maximum of the function $2^{-6}3^{-1}y^{2}-2^{-5}7^{-1}y^{3}$ for $1/2<y\leq1$ was determined. In fact, however,
one does not need the maximum but the minimum of this expression, which is attained for $y=1$. If we carry out the calculation in the correct way, we obtain the correct estimate \eqref{eq:corest2}.
But taking the maximum still gives the lower bound $\eqref{eq:corest1}$. The estimates in \cite{HM} for $d>2$ have also to be adjusted in the same way.

These lower bounds are also valid for the weighted $L_2$-discrepancy. In this paper we concentrate on the case of equal weights.
Here our aim is to improve these lower estimates. We show

\begin{thm}\label{thm:main}
 \begin{equation*}
   \bar{c}_2\geq0.0515599\ldots
 \end{equation*}
 That means: For every $N$-point set $\cP$ in $\I^{2}$ we have
 \begin{equation*}
  \norm\geq0.0515599\ldots \cdot \sqrt{\log N}.
 \end{equation*}
 Moreover,
 \begin{equation*}
  \bar{b}_{2} \geq 0.0610739\ldots.
 \end{equation*}
\end{thm}

The proof is a variant of the proof given in \cite{HM} but we go one step deeper in the analysis.

The paper is organized as follows. In Section \ref{sec:2}, we collect the necessary facts on
Haar-coefficients of the discrepancy function. In Section \ref{sec:3}, we describe the
steps of the proof omitting some technical details. These are given in Section \ref{sec:4}. 


\section{Haar-coefficients of the discrepancy function}\label{sec:2}

For $j=(j_1,j_2)\in\{-1,0,1,2,\ldots\}^{2}$ let $\bD_j$ be the set of $m=(m_1,m_2)\in\{0,1,\ldots\}^2$ with $0\leq m_i<2^{j_i}$ and let $|j|= \max(0,j_1)+\max(0,j_2)$. 
For such $j$ and $m$, we consider dyadic intervals
\begin{equation*}
 I_{j,m}:=\left[\frac{m_1}{2^{j_1}},\frac{m_1+1}{2^{j_1}}\right)\times \left[\frac{m_2}{2^{j_2}},\frac{m_2+1}{2^{j_2}}\right).
\end{equation*}
We denote by $I_{j,m}^{+,+}$ the left-lower, by $I_{j,m}^{+,-}$ the left-upper, by $I_{j,m}^{-,+}$ the right-lower, and by $I_{j,m}^{-,-}$ the right-upper quarter of $I_{j,m}$.
These are again dyadic intervals with a quarter of the area of $I_{j,m}$.
The Haar-functions $h_{j,m}$ are supported by $I_{j,m}$ and satisfy $h_{j,m}=1$ on $I_{j,m}^{+,+}$ and $I_{j,m}^{-,-}$, and $h_{j,m}=-1$ on $I_{j,m}^{+,-}$ and $I_{j,m}^{-,+}$.
Note that all dyadic intervals with fixed $j$ are congruent. Therefore we call $j$ the shape of the interval $I_{j,m}$.

Then by Parseval's equation we have
\begin{equation}\label{eq:pars}
 \norm^{2}=\sum_{j}2^{|j|}\sum_{m\in \bD_j}\left|\mu_{j,m}\right|^{2},
\end{equation}
where
\begin{equation*}
 \mu_{j,m}=\int_{\I^{2}}\D(x)\cdot h_{j,m}(x)\ud x
\end{equation*}
are the Haar-coefficients of the discrepancy function.

In \cite{HM} the above sum was estimated from below by summation over all $j$ and $m$ such that $I_{j,m}$ contains no point of $\cP$. For these $j$ and $m$ the Haar-coefficient $\mu_{j,m}$ has a very simple form. In this paper we also consider $j$ and $m$ such that $I_{j,m}$ contains exactly one point of $\cP$, and so we will obtain an improved lower bound.

The following two lemmas were already stated in \cite{H}:

\begin{lem}
 \label{lemma:muproduct}
 For $j\in\N_0^2$ and $m\in\bD_j$ we have
 \begin{equation*}
  \int_{\I^{2}}x_1x_2\cdot h_{j,m}(x_1,x_2)\ud x_1 \ud x_{2}= 2^{-2|j|-4}.
 \end{equation*}
\end{lem}

\begin{lem}
 \label{lemma:mucharempty}
 For $z=(z_1,z_2)\in\I^{2}$, let $C_{z}:=[z_1,1)\times[z_2,1)$ and let $\chi_{C_z}$ be the characteristic function of $C_z$. Let $j\in\N_0^2$ and $m\in\bD_j$ be such that $z\notin I_{j,m}$. Then
 \begin{equation*}
  \int_{\I^{2}}\chi_{C_z}(x)\cdot h_{j,m}(x)\ud x=0.
 \end{equation*}
\end{lem}

In the following we will also have to consider the case when $z\in I_{j,m}$:

\begin{lem}
 \label{lemma:mucharonepoint}
 With the notation of Lemma~\ref{lemma:mucharempty}, let now $z=(z_1,z_2)\in I_{j,m}$. Then
 \begin{equation*}
  \mu_{j,m}^z:=\int_{\I^{2}}\chi_{C_z}(x)\cdot h_{j,m}(x)\ud x=
  \left\{
  \begin{array}{rl}
   \left(z_1-\frac{m_1}{2^{j_1}}\right) \left(z_2-\frac{m_2}{2^{j_2}}\right) & \text{if } z\in I_{j,m}^{+,+},\\
   \left(z_1-\frac{m_1}{2^{j_1}}\right) \left(\frac{m_2+1}{2^{j_2}}-z_2\right) & \text{if } z\in I_{j,m}^{+,-},\\
   \left(\frac{m_1+1}{2^{j_1}}-z_1\right) \left(z_2-\frac{m_2}{2^{j_2}}\right) & \text{if } z\in I_{j,m}^{-,+},\\
   \left(\frac{m_1+1}{2^{j_1}}-z_1\right) \left(\frac{m_2+1}{2^{j_2}}-z_2\right) & \text{if } z\in I_{j,m}^{-,-}.\\
   \end{array}
\right.
 \end{equation*}
\end{lem}

The lemma is proved by simple calculations, see also Lemma 3.3 in \cite{H}.
As a corollary of the three lemmas we obtain immediately:

\begin{lem}\label{cor:haar}
For the Haar-coefficients of the discrepancy function $\D$ we have
\begin{equation*}
 \mu_{j,m}=
 \left\{
 \begin{array}{rl}
  -N2^{-2|j|-4} & \text{if }I_{j,m}\text{ contains no point from }\cP,\\
  \mu_{j,m}^z-N 2^{-2|j|-4} &\text{if }I_{j,m}\text{ contains exactly one point $z$ from }\cP.
 \end{array}
 \right.
\end{equation*}
\end{lem}


\section{Main steps of the proof}\label{sec:3}

In this section, we describe the high level structure of the proof. The details will then be presented in the next section.

Let $N$ be the number of points in $\cP$. Let $M\in\N_0$ and $0\leq\kappa<1$ be such that $N=2^{M+\kappa}$. 
For $j\in\N_0^2$, let $|j|=j_1+j_2$ be the level of the dyadic intervals $I_{j,m}$.
For $r\in \N_0$, let
$$ A_r(j) = \{ (j,m) \,:\, m \in\bD_j \text{ and }\# I_{j,m} \cap \cP = r \} \qquad \text{and}\qquad a_r(j)=\# A_r(j) $$
be the set of indices of intervals of shape $j$ containing exactly $r$ points of $\cP$ and its cardinality, respectively.
Observe that we have $2^{|j|}$ intervals for fixed $j$ and altogether $N=2^{M+\kappa}$ points, so
\begin{equation}\label{eq:condition_Arj}
   \sum_{r=0}^\infty a_r(j) = 2^{|j|} \qquad \text{and} \qquad \sum_{r=0}^\infty r a_r(j) = 2^{M+\kappa}.
\end{equation} 
Since the intervals of a given shape $j$ are mutually disjoint, this implies
\begin{equation}\label{eq:condition_Ar0}
   a_0(j) \ge 2^{|j|} - N = 2^{|j|} - 2^{M+\kappa}.
\end{equation} 
Finally, for $\ell\in\N_0$, let
$$ A_r(\ell) = \bigcup_{|j|=\ell}  A_r(j)  \qquad \text{and} \qquad a_r(\ell) = \# A_r(\ell) = \sum_{|j|=\ell}  a_r(j).$$
Since there are $\ell+1$ different shapes $j$ with level $|j|=\ell+1$, we obtain from \eqref{eq:condition_Arj} 
\begin{equation}\label{eq:condition_Arl}
   \sum_{r=0}^\infty a_r(\ell) = (\ell+1) 2^{\ell} \qquad \text{and} \qquad \sum_{r=0}^\infty r a_r(\ell) = (\ell+1) 2^{M+\kappa}.
\end{equation} 

We begin with recalling the main idea of the proof of the lower bound in \cite{HM}. Using in Parseval's equation \eqref{eq:pars} 
only the intervals containing no point of $\cP$ together with Lemma \ref{cor:haar} and \eqref{eq:condition_Ar0} we have
\begin{equation}\label{eq:HM}
 \norm^{2} \geq \sum_{\ell=0}^\infty  2^{\ell}  \sum_{(j,m)\in A_0(\ell)} \mu_{j,m}^2 \ge \sum_{\ell=M+1}^\infty (\ell+1) 2^{\ell} \left( 2^{\ell} - 2^{M+\kappa} \right)  2^{2M+2\kappa-4\ell-8}.
\end{equation}
This sum can then be evaluated and estimated.

To improve upon this estimate, it is not enough to consider further intervals separately since each Haar coefficient of an interval containing points of $\cP$ can be zero.
So we have to bundle together some Haar coefficients.
In our approach, we consider intervals $I_{j,m}$ of level $|j|=M$ or $|j|=M+1$ containing one point together with the two intervals of level $M+1$ and $M+2$, respectively, that are
contained in $I_{j,m}$ and contain the same point. More formally, for $(j,m) \in A_1(\ell)$, define
\begin{equation}\label{eq:rho}
   \rho_{j,m} =  \mu_{j,m}^2 +  \mu_{j',m'}^2 + \mu_{j'',m''}^2 
\end{equation} 
where $j'$ and $j''$ with $|j'|=|j''|=\ell+1$ are distinct, $I_{j',m'}, I_{j'',m''} \subset I_{j,m}$ and $I_{j,m} \cap \cP = I_{j',m'} \cap \cP = I_{j'',m''} \cap \cP$.
For $u=0,1,2$, let us call an interval $I_{j',m'}$ of level $\ell+1$ with $(j',m') \in A_1(\ell+1)$ {\em of type $u$}, if exactly $u$ of the two intervals $I_{j,m}$ of level $\ell$
containing it satisfy $(j,m) \in A_1(j)$. That means that exactly $2-u$ of these two intervals contain at least one additional point of $\cP$.
We also define for $\ell\ge 1$
$$ B_u(\ell) = \{ (j,m) \in A_1(\ell) \,:\, I_{j,m} \text{ is of type $u$}\} \qquad \text{and}\qquad b_u(\ell)=  \# B_u(\ell).$$
It follows from these definitions that
$$
   a_1(\ell+1) = b_0(\ell+1) + b_1(\ell+1) + b_2(\ell+1) \qquad \text{and}\qquad 2 a_1(\ell) = b_1(\ell+1)+2 b_2(\ell+1). 
$$ 
for $\ell\in\N_0$. 
Indeed, the first identity counts the intervals in level $\ell+1$ containing exactly one point of $\cP$ in two different ways.
For the second identity, consider the bipartite graph with intervals in level $\ell$ and $\ell+1$, respectively, containing exactly 
one point of $\cP$ as the two sets of vertices and draw an edge between an interval $I_{j,m}$ of level $\ell$ and an interval $I_{j',m'}$ of level $\ell+1$
if $I_{j',m'}$ is contained in $I_{j,m}$. Then the degree of each interval $I_{j,m}$ of level $\ell$ is exactly 2 and the degree of each interval $I_{j',m'}$ of level $\ell+1$
is $u$ if $(j',m')$ is of type $u$. Counting the edges in two different ways yields the second identity.
This further implies 
\begin{equation}\label{eq:b1}
    2 b_0(\ell+1) + b_1(\ell+1) = 2a_1(\ell+1)-2 a_1(\ell). 
\end{equation}

The crucial improvement of the estimate \eqref{eq:HM} is now 
 \begin{equation}\label{eq:main}
   \norm^{2} \geq \sum_{\ell=0}^\infty  2^{\ell}  \sum_{(j,m)\in A_0(\ell)} \mu_{j,m}^2 + \sum_{(j,m)\in A_1(M)} 2^M \rho_{j,m} + \sum_{u=0}^2 \sum_{(j,m)\in B_u(M+1)} (2-u) 2^M \rho_{j,m}.
 \end{equation}
Considering \eqref{eq:rho}, we see that this inequality indeed holds since 
each Haar coefficient of an interval in level $M$ gets at most a weight $2^M$,  
each Haar coefficient of an interval in level $M+2$ gets at most a weight $2^{M+1}$
and, by definition of $B_u(M+1)$, each Haar coefficient in level $M+1$ gets at most a weight $2\cdot 2^M=2^{M+1}$. 

To estimate the $\rho_{j,m}$ we find in the next section an explicit function $\gamma:[-1,1] \to \R_+$ such that the following lemma holds.
\begin{lem}\label{bound_rhojm}
 We have $2^{2M} \rho_{j,m} \ge  \gamma(\kappa)$ if $(j,m)\in A_1(M)$ and $2^{2M} \rho_{j,m} \ge 2^{-2} \gamma(\kappa-1)$ if $(j,m)\in A_1(M+1).$
\end{lem} 

Now it follows from \eqref{eq:main}, \eqref{eq:condition_Ar0}, Lemma \ref{cor:haar}, Lemma \ref{bound_rhojm} and \eqref{eq:b1} that
 \begin{equation}\label{eq:main1}
   \norm^{2} \geq 2^{-M} \Sigma_1 + \Sigma_2
 \end{equation}
with
 \begin{align}
   \Sigma_1 & = a_0(M) 2^{2\kappa-8} + a_0(M+1) 2^{2\kappa-11} + a_1(M) \gamma(\kappa) \nonumber\\
	          & \phantom{xxxxxxx}  + \big(b_1(M+1)+2 b_0(M+1) \big) 2^{-2} \gamma(\kappa-1) \label{eq:sigma1}\\
	          & = a_0(M) 2^{2\kappa-8} + a_1(M) \big(\gamma(\kappa)- 2^{-1} \gamma(\kappa-1) \big) \nonumber\\
						& \phantom{xxxxxxx}  + a_0(M+1) 2^{2\kappa-11} + a_1(M+1) \cdot 2^{-1} \gamma(\kappa-1) \nonumber
 \end{align}
and 
 \begin{equation}\label{eq:sigma_2}
   \Sigma_2 = \sum_{\ell=M+2}^\infty  2^{\ell}  (\ell+1) (2^{\ell}-2^{M+\kappa}) 2^{2M+2\kappa-4 \ell-8}.
 \end{equation}

Now the next Lemma shows that, under the conditions in \eqref{eq:condition_Arl}, $\Sigma_1$ is minimized if the points are
as evenly distributed in the boxes of level $M$ and $M+1$ as possible.

\begin{lem}\label{evendistribution}
 For $r\in\N_0$ and $\ell=M,M+1$, let $a_r(\ell)\in\N_0$ be such that the conditions in \eqref{eq:condition_Arl} are satisfied. Then we have
 \begin{eqnarray*}
   a_0(M) 2^{2\kappa-8} + a_1(M) \big(\gamma(\kappa)- 2^{-1} \gamma(\kappa-1) \big)  &\ge& (M+1)2^M \phi_M(\kappa) \\
	 a_0(M+1) 2^{2\kappa-11} + a_1(M+1) \cdot 2^{-1} \gamma(\kappa-1) &\ge& (M+2) 2^M \phi_{M+1}(\kappa)
 \end{eqnarray*}   
 with
 \begin{eqnarray*}
   \phi_M(\kappa)     &=&   (2-2^\kappa)\cdot \big(\gamma(\kappa) - 2^{-1} \gamma(\kappa-1) \big)  \\
	 \phi_{M+1}(\kappa) &=&    (2-2^\kappa)\cdot 2^{2\kappa-11} + 2^\kappa \cdot 2^{-1} \gamma(\kappa-1).  \\
 \end{eqnarray*}   
 \end{lem} 

Now \eqref{eq:main} and Lemma \ref{evendistribution} imply
 \begin{equation}\label{eq:main2}
   \norm^{2} \geq  \Sigma'_1 + \Sigma'_2
 \end{equation}
with
 \begin{align}
   \Sigma'_1 
	= (M+1) (2-2^\kappa)\cdot \big(\gamma(\kappa) - 2^{-1} \gamma(\kappa-1) \big) + (M+2)  2^\kappa \cdot 2^{-1} \gamma(\kappa-1)  \label{eq:sigma1p}
 \end{align}
and 
 \begin{equation}\label{eq:sigma_2p}
   \Sigma'_2 = \sum_{\ell=M+1}^\infty  2^{\ell}  (\ell+1) (2^{\ell}-N) 2^{2M+2\kappa-4 \ell-8} \ge (M+2) \left(   3^{-1} \cdot 2^{2\kappa-8}   -  7^{-1} \cdot 2^{3\kappa-8} \right).
 \end{equation}
The last inequality is the direct computation done in \cite{HM}.
Observe that the summation in $\Sigma'_2$ starts at $\ell=M+1$ instead of $\ell=M+2$ for $\Sigma_2$. We now integrated the summand corresponding to the intervals not containing points of $\cP$
into this sum. This has the advantage that this sum is exactly the same sum already estimated in \cite{HM}. The complete improvement is now contained in $\Sigma'_1$.

Altogether, we obtain the improved bound
$$ \norm^{2} \geq (M+1) \Delta(\kappa) \ge \log N \cdot \frac{\Delta(\kappa)}{\log 2}$$
with
\begin{equation} \label{eq:Delta}
 \Delta(\kappa) = 3^{-1} \cdot 2^{2\kappa-8}   -  7^{-1} \cdot 2^{3\kappa-8} + (2-2^\kappa) \gamma(\kappa) + (2^\kappa-1) \gamma(\kappa-1).
\end{equation}
A final analysis shows that
\begin{equation}\label{analysis_functions}
 \inf_{\kappa\in [0,1)} \sqrt{\frac{\Delta(\kappa)}{\log 2}} = 0.0515599\ldots
 \ \text{ and } \  
 \sup_{\kappa\in [0,1)} \sqrt{\frac{\Delta(\kappa)}{\log 2}} = 0.0610739\ldots
\end{equation}
and finishes the proof of Theorem \ref{thm:main}. 
  

\section{Details of the proof}\label{sec:4}

In this section, we prove  Lemmas \ref{bound_rhojm} and \ref{evendistribution} and provide the analysis showing \eqref{analysis_functions}.

\begin{proof}[Proof of Lemma \ref{bound_rhojm}]
Let $(j,m) \in A_1(M)$ and assume that the point $p=(p_1,p_2)$ of $\cP$ which is in $I_{j,m}$ is contained in the left lower part of $I_{j,m}$, 
i.e., with $j^{\prime}=(j_{1}+1,j_2)$ and $j^{\prime\prime}=(j_1,j_2+1)$ we have $m^{\prime}=(2m_1,m_2)$ and $m^{\prime\prime}=(m_1,2m_2)$. 
The three other cases are treated in exactly the same way.
 
 This means that
 \begin{equation*}
  \mu_{j,m}=xy-\frac{2^{M+\kappa}}{2^{2j_1+2j_2+4}},
 \end{equation*}
where $x:=p_1-\frac{m_1}{2^{j_1}}$ and $y:=p_2-\frac{m_2}{2^{j_2}}$ and hence $0\leq x<\frac{1}{2^{j_1+1}}$ and $0\leq y<\frac{1}{2^{j_2+1}}$.

We have now to distinguish between four cases:
\begin{align*}
  \text{Case 1:} &&  0\leq x<\frac{1}{2^{j_1+2}}, \quad &&	0\leq y < \frac{1}{2^{j_2+2}}\\
  \text{Case 2:} &&	0\leq x<\frac{1}{2^{j_1+2}},\quad	&&	\frac{1}{2^{j_2+2}}\leq y < \frac{1}{2^{j_2+1}}\\
  \text{Case 3:} &&	\frac{1}{2^{j_1+2}}\leq x<\frac{1}{2^{j_1+1}},\quad	&&	0\leq y < \frac{1}{2^{j_2+2}}\\
  \text{Case 4:} &&	\frac{1}{2^{j_1+2}}\leq x<\frac{1}{2^{j_1+1}},\quad	&&	\frac{1}{2^{j_2+2}}\leq y < \frac{1}{2^{j_2+1}}
\end{align*}

\noindent{\bf Case 1:} In this case, by Lemma~\ref{lemma:mucharonepoint}, we have
$$  \rho_{j,m} = \left( xy-\frac{2^{M+\kappa}}{2^{2j_1+2j_2+4}} \right)^{2} + 2\left( xy-\frac{2^{M+\kappa}}{2^{2j_1+2j_2+6}} \right)^{2} = 
     \left( z-\frac{1}{2^{M+4-\kappa}} \right)^{2} + 2\left( z-\frac{1}{2^{M+6-\kappa}} \right)^{2}$$

with $0\leq z:=x y<\frac{1}{2^{M+4}}$. Simple analysis shows that $\rho_{j,m}$ attains its minimum $\frac{3}{2^{2M+11-2\kappa}}$ for $z=\frac{1}{2^{M+5-\kappa}}$.

\noindent{\bf Case 2:} In this case, again by Lemma \ref{lemma:mucharonepoint}, we have
$$  \rho_{j,m} = \left( xy-\frac{1}{2^{M+4-\kappa}} \right)^{2} + \left( xy-\frac{1}{2^{M+6-\kappa}} \right)^{2}+ 	\left( x\left( \frac{1}{2^{2j_2+1}}-y \right) -\frac{1}{2^{M+6-\kappa}} \right)^{2}.$$
Hence
\begin{eqnarray*}
 \left(2^{j_1+j_2+2}\right)^{2} \rho_{j,m} &=& 2^{2M+4}\rho_{j,m}\\
 &=& \left(\alpha\beta-\frac{1}{2^{2-\kappa}}\right)^{2}+  \left(\alpha\beta-\frac{1}{2^{4-\kappa}}\right)^{2} +\left( \alpha(1-\beta)-\frac{1}{2^{4-\kappa}} \right)^{2} =:f(\alpha,\beta)
\end{eqnarray*}
with $0\leq\alpha:=2^{j_1+1} x<\frac{1}{2}$ and $\frac{1}{2}\leq\beta:=2^{j_2+1} y <1.$

Simple analysis of $f(\alpha,\beta)$ shows that $f$ in this region for $\alpha$ and $\beta$ has minimal value $\frac{9\cdot2^{2\kappa}}{512}$ for $\alpha=\frac{7}{32}\cdot2^{\kappa}$ and $\beta=\frac{5}{7}$.
Hence, in Case 2, we always have
\begin{equation*}
 \rho_{j,m}\geq \frac{1}{2^{2M+4}}\cdot\frac{9\cdot2^{2\kappa}}{512}.
\end{equation*}

\noindent{\bf Case 3:} This case is treated in the same way as Case 2 and also results in
\begin{equation*}
 \rho_{j,m}\geq \frac{1}{2^{2M+4}}\cdot \frac{9\cdot2^{2\kappa}}{512} \quad\text{ always.}
\end{equation*}

\noindent{\bf Case 4:} In this case, again by Lemma~\ref{lemma:mucharonepoint}, we have
$$  \rho_{j,m} =  \left( xy-\frac{1}{2^{M+4-\kappa}} \right)^{2} + \left( x\left(\frac{1}{2^{j_2+1}}-y\right)-\frac{1}{2^{M+6-\kappa}} \right)^{2}+\left( \left(\frac{1}{2^{j_1+1}}- x\right)y -\frac{1}{2^{M+6-\kappa}} \right)^{2}.$$
Hence
\begin{multline*}
  2^{2M+4}\rho_{j,m}
 = \left(\alpha\beta-\frac{1}{2^{2-\kappa}}\right)^{2}+  \left(\alpha(1-\beta)-\frac{1}{2^{4-\kappa}}\right)^{2} 
 +\left( (1-\alpha)\beta-\frac{1}{2^{4-\kappa}} \right)^{2} =:g(\alpha,\beta),
\end{multline*}
now with
\begin{equation*}
 \frac{1}{2}\leq\alpha:=2^{j_1+1}x<1 \qquad \text{and} \qquad \frac{1}{2}\leq\beta:=2^{j_2+1}y<1.
\end{equation*}
The analysis of $g(\alpha,\beta)$, i.e., the determination of $\min_{\alpha,\beta}g(\alpha,\beta)$ is explicitly possible, but leads to an equation of degree 4. The analysis was carried out with the help of Mathematica. We will not give the details of this in principle elementary but elaborate analysis.
 
 It turns out that $g$
attains its minimum for each $\kappa$ with $0\leq\kappa<1$ for some $\alpha=\beta=\beta(\kappa)$. 
The function $\beta(\kappa)$ can be explicitly expressed with fourth-roots. Hence,
\begin{equation*}
 h(\kappa):=\min_{\alpha,\beta}g(\alpha,\beta)=g(\beta(\kappa),\beta(\kappa))
\end{equation*}
for each $\kappa$ is an explicitly given (but ugly) function. Its form is given by the dashed line in Figure~1.

Hence, in Case~4,
\begin{equation*}
 \rho_{j,m}\geq\frac{1}{2^{2M+4}}\cdot h(\kappa).
\end{equation*}

So, altogether we have in any case 
$$ 2^{2M} \rho_{j,m} \ge \min \big( 3\cdot 2^{2\kappa-11}, 9\cdot 2^{2\kappa-13}, 2^{-4} h(\kappa) \big) = \min \big( 9\cdot 2^{2\kappa-13}, 2^{-4} h(\kappa) \big) =:\gamma(\kappa). $$
So, $\gamma(\kappa)$ is the minimum of the two functions $2^{-4} h(\kappa)$ (dashed line) and $9\cdot 2^{2\kappa-13}$ (undashed line) in Figure 1 and it is explicitly given.
\begin{figure}[!htbp] \label{fig1}
 \centering
 \includegraphics[width=0.5\textwidth]{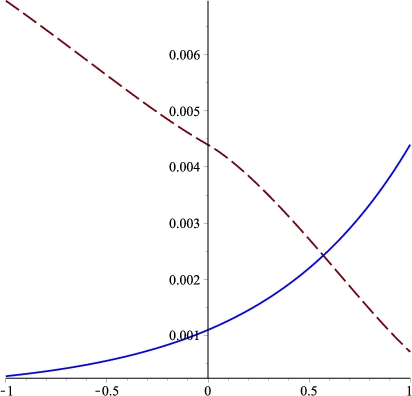}
 \caption{The functions $2^{-4} h(\kappa)$ and $9\cdot 2^{2\kappa-13}$}
\end{figure}

The inequality for $(j,m)\in A_1(M+1)$ immediately follows by noting that we can proceed in the same way as before by replacing $\kappa$ by $\kappa-1$ and $M$ by $M+1$ now.
\end{proof}

\begin{proof}[Proof of Lemma \ref{evendistribution}]
 For $r\in \N_0$, let $a_r \ge 0$ be such that
 \begin{equation}\label{eq:aux}
  \sum_{r\ge 0} a_r =1 \qquad \text{and} \qquad \sum_{r \ge 0} r a_r = \sigma.
 \end{equation}
 Furthermore, let $\alpha,\beta \ge 0$ be such that $\alpha \ge 2 \beta$.
 Then we have
 \begin{align}
   \alpha a_0 + \beta a_1 \ge \alpha(1-\sigma) + \beta \sigma & \qquad \text{if } 0\le \sigma\le 1 \label{eq:aux1}\\
	 \alpha a_0 + \beta a_1 \ge \beta (2-\sigma)                & \qquad \text{if } 1\le \sigma\le 2 \label{eq:aux2}.
 \end{align}
Indeed, \eqref{eq:aux1} follows directly from the identity
$$  \alpha a_0 + \beta a_1 = \alpha(1-\sigma) + \beta \sigma + \sum_{r\ge 2} \left( \alpha (r-1) - \beta r \right) a_r$$
and $\alpha(r-1) - \beta r \ge \beta (r-2) \ge 0$ for $r\ge 2$.
Similarly, \eqref{eq:aux1} follows from the identity
$$   2 a_0 +  a_1 = 2-\sigma + \sum_{r\ge 3} (r-2) a_r $$
together with $\alpha a_0 + \beta a_1 \ge \beta(2 a_0+a_1)$.

Now the inequalities to be proved follow from \eqref{eq:aux1} and \eqref{eq:aux2} by considering the special cases
$$ a_r = \big( (M+1) 2^M \big)^{-1}     a_r(M), \  \sigma=2^\kappa     \in [1,2],\   \alpha=2^{2\kappa-8},\   \beta = \gamma(\kappa) - 2^{-1} \gamma(\kappa-1) $$
and
$$ a_r = \big( (M+2) 2^{M+1} \big)^{-1} a_r(M+1),  \ \sigma=2^{\kappa-1} \in [0,1], \ \alpha=2^{2\kappa-11},  \ \beta =  2^{-1} \gamma(\kappa-1). $$	
Indeed, the equalities \eqref{eq:aux} follow from \eqref{eq:condition_Arl}. Moreover, in the first case
$$ 2 \beta = 2 \gamma(\kappa) -  \gamma(\kappa-1) \le 2 \gamma(\kappa) \le 9\cdot 2^{2\kappa-12} < 2^{2\kappa-8} = \alpha  $$
and in the second case
$$ 2 \beta = \gamma(\kappa-1) \le 9\cdot 2^{2\kappa-15} < 2^{2\kappa-11} = \alpha$$
also hold. Finally, $\beta\ge 0$ is obvious in the second case and is easily checked in the first case.
\end{proof}

\begin{proof}[Proof of \eqref{analysis_functions}]
\begin{figure}[!htbp]
 \centering
 \includegraphics[width=0.5\textwidth]{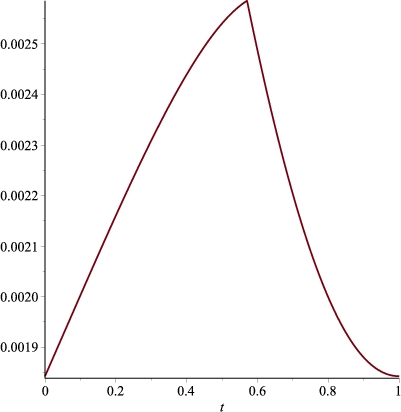}
 \caption{The function $\Delta(\kappa)$}
\end{figure}
It is an elementary task to analyze the function $\Delta(\kappa)$ on the unit interval. 
With the help of Mathematica and basic differential calculus we see 
that there is a $\kappa_0\in(0,1)$ with $\Delta$ monotonically increasing in $[0,\kappa_0)$ and monotonically decreasing in $(\kappa_0,1]$. 
Moreover, we have
$$ \Delta(0)=\Delta(1)=3^{-1}\cdot 2^{-6} - 7^{-1}\cdot 2^{-5} +\gamma(0)=\frac{317}{172032}.$$
Hence,
\begin{equation*}
 \min_{\kappa\in[0,1]}\Delta(\kappa)=\min\big(\Delta(0),\Delta(1)\big)=\frac{317}{172032}=0.00184268\ldots
\end{equation*}
(see the illustration of $\Delta(\kappa)$ in Figure~2) and consequently
\begin{equation*}
  \inf_{\kappa\in [0,1)} \sqrt{\frac{\Delta(\kappa)}{\log 2}} = 0.0515599\ldots
\end{equation*}
Moreover, numerical computation gives $\kappa_0=0.5705243\ldots$ and $\Delta(\kappa_0)=0.00258545\ldots$ which finally leads to
\begin{equation*}
  \sup_{\kappa\in [0,1)} \sqrt{\frac{\Delta(\kappa)}{\log 2}} = \sqrt{\frac{\Delta(\kappa_0)}{\log 2}} = 0.0610739\ldots
\end{equation*}
\end{proof}

We close with some remarks and open problems. The lower bounds in \cite{HM} are valid for
the weighted discrepancy. This is mainly due to the fact that only dyadic boxes are used that do not
contain points of the pointset. To show the better bounds in this paper we need to consider also dyadic
boxes containing points. This does not seem to easily generalize to the case of the weighted discrepancy.
Also, to extend the approach to dimension $d>2$ seems to be technically difficult since there are 
more cases to consider than in the proof for $d=2$.

It is also an open problem if possibly $  \bar{a}_{d} = \bar{b}_{d}$,
i.e. if the limit 
\begin{equation*}
 \lim_{N\to\infty} \inf_{\# \cP = N} \norm\cdot\left(\log N\right)^{-\frac{d-1}{2}}
\end{equation*}
exists.


\end{document}